\documentclass[12pt]{article}
\usepackage[utf8]{inputenc}
\usepackage{amssymb, amsmath,amsthm, thmtools, mathabx, mathtools}
\usepackage{bbm}
\usepackage{amsfonts}
\usepackage[left=2.5cm,right=2.5cm,top=2cm,bottom=2cm]{geometry}
\usepackage{xcolor}

\usepackage{tabto}
\TabPositions{2 cm, 4 cm, 6 cm, 8 cm}

\usepackage{tikz-cd}

\usepackage{tikz}
\usetikzlibrary{arrows, automata, positioning}

\usepackage{tocloft}
\usepackage{appendix}

\usepackage{enumerate}

\usepackage{hyperref}
\usepackage[capitalize]{cleveref}

\hypersetup{
  colorlinks   = true, 
  urlcolor     = blue, 
  linkcolor    = blue, 
  citecolor   = blue 
}
 
 \newtheorem{thmA}{Theorem}

\newtheorem{propA}[thmA]{Proposition}
\newtheorem{corA}[thmA]{Corollary}

\newtheorem{quA}[thmA]{Question}

\newtheorem{theorem}{Theorem}[section]

\newtheorem{lemma}[theorem]{Lemma}

\newtheorem{proposition}[theorem]{Proposition}

\newtheorem*{theorem*}{Theorem}

\theoremstyle{definition}

\newtheorem{remark}[theorem]{Remark}

\newtheorem{question}[theorem]{Question}

\newtheorem*{remark*}{Remark}
\newtheorem*{notation*}{Notation}
\newtheorem*{acks*}{Acknowledgements}
\newtheorem*{out*}{Outline}

\renewcommand\leq{\leqslant}

\newcommand{\HHH}{\operatorname{H}}

\newcommand{\Z}{\mathbb{Z}}

\newcommand{\C}{\mathbb{C}}

\newcommand{\T}{\mathrm{(T)}}
\newcommand{\TT}{\mathrm{(T_2)}}
\newcommand{\TTT}{\mathrm{[T_2]}}

\newcommand{\normal}[1]{\left<\! \left< #1\right> \!\right>}

\newcommand{\G}{\mathcal{G}}

\newcommand{\Un}{\mathrm{U}(n)}

\begin{document}

\title{Stability, approximable quotients, \\ and higher property (T)}
\author{Francesco Fournier-Facio}
\date{\today}
\maketitle

\begin{abstract}
We construct a wealth of groups that are finitely presented, Frobenius stable, have property $\T$, but are very far from having property $\TT$. Our method also shows that property $\TT$ does not pass to quotients. A post-scriptum relates and comments on an experience this paper had with an AI benchmark.
\end{abstract}

\section{Introduction}

Let $\G$ be a family of groups equipped with bi-invariant metrics, and let $\Gamma$ be a countable group. An \emph{asymptotic homomorphism} is a sequence of maps $\phi_n \colon \Gamma \to (G_n, d_n) \in \G$ such that
\[d_n(\phi_n(gh), \phi_n(g)\phi_n(h)) \to 0 \text{ for all } g, h \in \Gamma.\]
We say that $\Gamma$ is \emph{$\G$-stable}\footnote{here we are only concerned with \emph{pointwise} stability, as opposed to more classical notions of \emph{uniform} stability \cite{kazhdan}} if every asymptotic homomorphism $(\phi_n)_n$ is \emph{close} to a sequence of genuine homomorphisms $(\psi_n)_n$, that is
\[d_n(\phi_n(g), \psi_n(g)) \to 0 \text{ for all } g \in \Gamma.\]
Stability has received much attention in the past few years, especially when $\G$ is the family of symmetric groups with the normalised Hamming distance \cite{AP, BLT}, or the family of finite-dimensional unitary groups with a bi-invariant matrix norm, such as the Hilbert--Schmidt \cite{HS, LV} or operator norm \cite{voiculescu, dadarlat}.

In this note we are mostly concerned with the family of finite-dimensional unitary groups equipped with the \emph{Frobenius norm}, which is the Euclidean norm induced by the embedding $\Un \hookrightarrow \mathbb{C}^{n \times n}$. This was first studied by De Chiffre--Glebsky--Lubotzky--Thom, who proved that if $\Gamma$ is a finitely presented group such that $\HHH^2(\Gamma; V) = 0$ for every unitary $\Gamma$-representation $V$, then $\Gamma$ is Frobenius stable \cite{DCGLT}. This applies to finitely generated virtually free groups, and more interestingly, it applies to groups with higher property $\T$. Following Bader--Sauer \cite{T2}, we say that $\Gamma$ has \emph{property $\TTT$} if $\HHH^i(\Gamma; V) = 0$ for every unitary $\Gamma$-representation $V$ and $i \in \{1, 2\}$. In \cite{DCGLT}, this is used to prove Frobenius stability of some higher rank $p$-adic lattices and their finite central extensions.

More examples of Frobenius stable groups arose recently in work of Bader--Lubotzky--Sauer--Weinberger \cite{BLSW}. These are higher rank lattices that need not have property $\TTT$, but the only obstruction comes from trivial coefficients: these lattices still have the more permissive \emph{property $\TT$}, meaning that the second cohomology vanishing holds under the additional assumption that $V^{\Gamma} = 0$. To the best of our knowledge, these are all known examples of Frobenius stable groups, in fact Frobenius stability seems to be a rather special property: see \cite{ESS}\footnote{all of the examples of operator unstable groups from \cite{ESS} also witness Frobenius instability, thanks to the inequalities $\| \cdot \|_{\mathrm{op}} \leq \| \cdot \|_2 \leq \sqrt{n}\| \cdot \|_{\mathrm{op}}$}, \cite[Section 2]{DCGLT} and \cite{glebe1, glebe2}. We present new examples, which have property $\T$ (separating them from virtually free groups), but are very far from having property $\TT$ (separating them from higher rank lattices).

\begin{thmA}
\label{main}

Every countable (recursively presented) group embeds into a (finitely presented) group $\Gamma$ with property $\T$, that is Frobenius stable but does not have property $\TT$. In fact, $\HHH^2(\Gamma; V) \neq 0$ for \emph{every} unitary $\Gamma$-representation $V$.
\end{thmA}

A group is \emph{$\G$-approximable} if it admits an asymptotic homomorphism $(\phi_n)_n$ such that:
\[\liminf\limits_{n \to \infty} d_n(\phi_n(g), 1_{G_n}) > 0 \text{ for all } g \in \Gamma \setminus \{1_{\Gamma}\}.\]
Many important properties in group theory are $\G$-approximability properties for various families $\G$, for example soficity \cite{weiss}, hyperlinearity \cite{radulescu} and matricial finiteness \cite{MF}. In these cases, it is an open problem whether there exists a non-approximable group. The main exception is the Frobenius norm \cite{DCGLT, BLSW} or more generally the unnormalised Schatten $p$-norm for $p \in [1, \infty)$ \cite{p:approx, 1:approx}, where non-approximable groups are known to exist.

It is easy to see that every asymptotic homomorphism is close to one that factors through an approximable quotient \cite[Section 2.3]{DCGLT}. Therefore groups with no non-trivial approximable quotients are (vacuously) stable. \cref{main} exploits this observation.

\begin{propA}
\label{no approx quotients}
Suppose that there exists a non-$\G$-approximable group $A$. Then for every countable group $C$, there exists a group $\Gamma$ that is hyperbolic relative to a subgroup $K$ containing a copy of $C$, such that $\Gamma$ has property $\T$, and no non-trivial $\G$-approximable quotient. Moreover, if $A$ has solvable word problem and $C$ is recursively presented, then $\Gamma$ can be chosen to be finitely presented.
\end{propA}

This is in the spirit of several results in geometric group theory that, from the existence of a non-residually finite hyperbolic group, deduce the existence of hyperbolic groups with more striking properties (e.g. no finite quotient \cite{kapovichwise}, or even no amenable quotient \cite{arzhantseva}). Groups with no approximable quotients with a subset of the above properties can be built more easily (see \cref{easier}), however we are interested in preserving at once finite presentability, since this is the setting in which cohomological methods can be applied, and relative hyperbolicity, which we will need later to apply \cref{one rel quotients}.

Moreover, we insist on imposing property $\T$ because of the special relation that this property has with questions of stability and approximability. Becker--Lubotzky showed that if $\Gamma$ is an infinite sofic group with property $\T$, then it is not permutation stable \cite{BL:T}. \cref{no approx quotients} shows that the soficity hypothesis is either vacuous, if all groups are sofic, or necessary. The same holds for hyperlinearity and Hilbert--Schmidt stability, and also for the local versions \cite{localP, localHS}. In another direction, some weak forms of stability for certain property $\T$ groups imply the existence of non-sofic or non-hyperlinear groups \cite{burtonbowen, alon, alon:lattices, chapmanpeled}; \cref{no approx quotients} can be seen as a sort of converse of this philosophy.

\medskip

To go from \cref{no approx quotients} to \cref{main}, we only need to ensure that the groups constructed enjoy the strong negation of property $\TT$.

\begin{propA}
\label{one rel quotients}

Let $\Gamma$ be a (finitely presented) infinite group with property $\T$, that is hyperbolic relative to a subgroup $K$. Then there exists a (finitely presented) quotient $\pi \colon \Gamma \to \bar{\Gamma}$ that has finite kernel when restricted to $K$, such that $\bar{\Gamma}$ is hyperbolic relative to $\pi(K)$, and does not have property $\TT$. In fact, $\HHH^2(\bar{\Gamma}; V) \neq 0$ for every unitary $\bar{\Gamma}$-representation $V$.
\end{propA}

The non-vanishing holds beyond unitary representations, for instance it holds for $L^1$-coefficients, which is particularly interesting in view of upcoming work of Bader--Bader--Bader--Sauer \cite{L1:lattices}, see \cref{L1}.

\begin{proof}[Proof of \cref{main}]
Let $C$ be a countable (recursively presented) group. Up to taking a free product with $\mathbb{Z}$, may assume that $C$ has no non-trivial finite normal subgroups. There exists a non-Frobenius approximable group $A$ that is a finite central extension of a finitely presented linear group \cite{DCGLT}. Thus $A$ is an extension of finitely presented groups with solvable word problem, hence it has solvable word problem. Now combining \cref{no approx quotients} and \cref{one rel quotients}, we obtain a (finitely presented) group $\Gamma$ that has property $\T$, verifies the strong negation of property $\TT$, and has no Frobenius approximable quotients, in particular it is Frobenius stable.
\end{proof}

\cref{one rel quotients} has the following additional consequence of independent interest, which shows how differently property $\TT$ behaves from property $\T$\footnote{this was announced in \cite[Remark 57]{ICM}}.

\begin{corA}
\label{quotients not T2}

There exists a hyperbolic group with $\TTT$ with a hyperbolic quotient that does not have $\TT$. Hence, $\TTT$ and $\TT$ do not pass to quotients.
\end{corA}

\begin{proof}
Cocompact lattices in the Cayley plane are hyperbolic and have $\TT$: combine \cite[Theorem B and Appendix A]{T2} and \cite[Lemma 3.13]{T2}. Their second Betti number vanishes, because they are fundamental groups of non-Hermitian locally symmetric spaces of rank one \cite{matsushima}, so they also have $\TTT$. \cref{one rel quotients} gives quotients that remain hyperbolic but do not have $\TT$.
\end{proof}

All of our stable groups have no approximable quotients, in particular they have no finite quotients, so the following remains open.

\begin{quA}
Does there exist a (finitely presented) \emph{residually finite} Frobenius stable group that does not have $\TT$?
\end{quA}

A good candidate could be $\mathrm{SL}_3(\mathbb{Z})$, which does not have $\TT$ \cite{SL3} and is not known to be Frobenius stable (the results from \cite{BLSW} assume rank at least three).

\begin{acks*}
The author is supported by the Herchel Smith Postdoctoral Fellowship Fund. He thanks Segev Gonen Cohen for asking whether Frobenius stable groups have property $\TT$, and Saar Bader, Shaked Bader, Uri Bader and Roman Sauer for useful discussions around \cref{quotients not T2}. He also thanks Shaked Bader, Alon Dogon, Forrest Glebe, Roman Sauer and Andreas Thom for comments on an early draft.
\end{acks*}

\section{No approximable quotients}

\begin{lemma}
\label{normalgen}

Suppose that $\Gamma$ has a non-$\G$-approximable subgroup $A$ with the property that every non-identity element in $A$ normally generates $\Gamma$. Then $\Gamma$ has no non-trivial $\G$-approximable quotient.
\end{lemma}

\begin{proof}
Suppose that $\Gamma$ has a $\G$-approximable quotient $\pi \colon \Gamma \to \bar{\Gamma}$. Because $A$ is not $\G$-approximable, it cannot embed into $\bar{\Gamma}$, so the kernel of $\pi$ must intersect $A$ non-trivially. By the assumption, $\bar{\Gamma}$ must be trivial.
\end{proof}

\begin{remark}
\label{easier}

\cref{normalgen} applies when $\Gamma$ is simple. This allows to easily construct many groups with no approximable quotients, since every countable group embeds into a finitely generated simple group $\Gamma$ \cite{hall, gorjuskin}, and in fact $\Gamma$ can be chosen to have property $\T$ \cite[Theorem 4.1]{charquot}. However, such a group might have property $\TT$, and simplicity means that we cannot perform a construction analogous to \cref{one rel quotients} to obstruct it. Moreover, whether $\Gamma$ can be chosen to be finitely presented depends on whether $A$ satisfies the Boone--Higman Conjecture \cite{BH}. This is known for many linear groups \cite{scott, selfsim}, but to the best of our knowledge it is open for their (non-approximable) finite central extensions.
\end{remark}

Next we need an ingredient about quotients of relatively hyperbolic groups, which comes from small cancellation theory \cite{osin}\footnote{one could also work in the more general framework of small cancellation theory over acylindrically hyperbolic groups \cite{hull}}.

\begin{proposition}
\label{sc:product}

Let $\Lambda$ be a non-elementary hyperbolic group with no non-trivial finite normal subgroups, and let $K$ be a (finitely presented) group. Let $g_1, \ldots, g_k \in \Gamma \coloneqq \Lambda \ast K$. Then there exists a (finitely presented) quotient $\pi \colon \Gamma \to \bar{\Gamma}$ that is injective on $K$, such that $\bar{\Gamma}$ is hyperbolic relative to $\pi(K)$, has no non-trivial finite normal subgroups, and such that $\pi(g_i) \in \pi(\Lambda)$ for $i = 1, \ldots, k$.
\end{proposition}

\begin{proposition}
\label{sc:normal}

Let $\Gamma$ be a (finitely presented) group that is hyperbolic relative to a subgroup $K$, and suppose moreover that $\Gamma$ has no non-trivial finite normal subgroups. Let $g_1, \ldots, g_k \in \Gamma$ and let $\Lambda < \Gamma$ be a non-trivial normal subgroup. Then there exists a (finitely presented) quotient $\pi \colon \Gamma \to \bar{\Gamma}$ that is injective on $K$, such that $\bar{\Gamma}$ is hyperbolic relative to $\pi(K)$, has no non-trivial finite normal subgroups, and $\pi(g_i) \in \pi(\Lambda)$ for $i = 1, \ldots, k$.
\end{proposition}

\begin{proof}
Both propositions are an application of \cite[Theorem 2.4]{osin}, which gives the statement more generally when $\Gamma$ is hyperbolic relative to $K$ and $\Lambda$ is \emph{suitable}. This is defined in \cite[Definition 2.2]{osin}; equivalently \cite[Definition 3.21]{CIOS}: $\Lambda$ is suitable if it is not virtually cyclic, does not normalise any non-trivial finite normal subgroup of $\Gamma$, and contains an element that is not conjugate into $K$.

Now \cref{sc:product} follows, as $\Lambda$ is easily seen to be a suitable subgroup of $\Gamma \coloneqq \Lambda \ast K$, which is hyperbolic relative to $K$. \cref{sc:normal} follows from the fact that in relatively hyperbolic groups with no non-trivial finite normal subgroups, every non-trivial normal subgroup is suitable \cite[Lemma 3.23]{CIOS}. The proof of \cite[Theorem 2.4]{osin} produces the quotient by adding $k$ small cancellation relations, one for each element $g_i$, in particular the construction preserves finite presentability.
\end{proof}

\begin{proof}[Proof of \cref{no approx quotients}]
Let $A$ be a non-$\G$-approximable group, and let $C$ be a countable group. We choose a group $H$ that contains $C$ and is finitely generated \cite{HNN}, or finitely presented in case $C$ is recursively presented \cite{higman}. If $A$ has solvable word problem, then we choose a finitely presented group $P$ containing a simple subgroup $S$ containing $A$ \cite{BH}; otherwise we choose $P = S$ to be a finitely generated simple group containing a copy of $A$ (see \cref{easier}). All of the following constructions will preserve finite generation and finite presentability.

Let $K \coloneqq H \ast P$, and let $\Lambda$ be a hyperbolic group with property $\T$ and no non-trivial finite normal subgroups (such as the one used for \cref{quotients not T2}), and let $\Gamma \coloneqq \Lambda \ast K$. Applying \cref{sc:product} with $\{g_1, \ldots, g_k\}$ a generating set for $K$, we obtain a quotient $\pi_1 \colon \Gamma \to \Gamma_1$ such that the restriction $\pi|_\Lambda \colon \Lambda \to \Gamma_1$ is surjective, hence $\Gamma_1$ has property $\T$. Now apply \cref{sc:normal} with $\Lambda$ the normal closure of $\pi_1(S)$ and $\{ g_1, \ldots, g_k \}$ a generating set for $\Gamma_1$. This gives a quotient $\pi_{1, 2} \colon \Gamma_1 \to \Gamma_2$, such that if we denote by $\pi_2 \colon \Gamma \to \Gamma_2$ the composition, $\pi_2(S)$ normally generates $\Gamma_2$. Since $\pi_2(S) \cong S$ is moreover simple, every non-identity element of $\pi_2(S)$ normally generates $\Gamma_2$. Moreover $\pi_2(S)$ contains the non-approximable subgroup $\pi_2(A) \cong A$, hence by \cref{normalgen}, we conclude that $\Gamma_2$ has no approximable quotients. Finally, $\pi_2$ is injective on $K$, hence $C$ embeds into $\pi_2(K)$, the parabolic subgroup of $\Gamma_2$.
\end{proof}

\begin{remark}
The hypothesis on $A$ having solvable word problem was used in order to apply the Boone--Higman Theorem: there are embeddings $A \hookrightarrow S \hookrightarrow P$ where $S$ is simple and $P$ is finitely presented \cite{BH}. The argument would work just as well assuming only that $A$ embeds into a finitely presented group $P$ with the property that $A$ is contained in the normal closure of every $1 \neq a \in A$ in $P$. But even this weaker condition implies that $A$ has solvable word problem: Kuznetsov's algorithm applies verbatim \cite{kuz}.

It is an interesting question whether the existence of a non-sofic group implies the existence of a non-sofic group with solvable word problem. Note that groups with solvable problems are \emph{not} dense in the space of marked groups \cite[Proposition 6]{isolated}.
\end{remark}

\section{Quotients without higher property (T)}

We need another ingredient about quotients of relatively hyperbolic groups, this time coming from group-theoretic Dehn filling \cite{osin2}\footnote{again one could work in the more general framework of acylindrically hyperbolic groups, which from this point of view are groups with a non-degenerate hyperbolically embedded subgroup \cite{DGO}}. Given a group $\Gamma$, a subgroup $\Lambda < \Gamma$, and a normal subgroup $N$ of $\Lambda$, we say that $(\Gamma, \Lambda, N)$ is a \emph{Cohen--Lyndon triple} if
\[\normal{N} = \ast_{t \in T} tNt^{-1},\]
where $\normal{N}$ denotes the normal closure of $N$ in $\Gamma$, and $T$ is a set of transversals of $\Lambda\normal{N}$ in $\Gamma$. This is named after Cohen and Lyndon who established this property for cyclic subgroups of the free group \cite{CL}, and it is a common feature of group-theoretic Dehn fillings in negative curvature \cite{CL0, CL1}.

\begin{lemma}
\label{find CL}

Let $\Gamma$ be a group that is non-elementary hyperbolic relative to a subgroup $K$, and suppose moreover that $\Gamma$ has no non-trivial finite normal subgroups. Then there exists an infinite cyclic subgroup $\Lambda$ such that $(\Gamma, \Lambda, \Lambda)$ is a Cohen--Lyndon triple. Moreover, the quotient $\pi \colon \Gamma \to \Gamma / \normal{\Lambda}$ is injective on $K$, and $\Gamma / \normal{\Lambda}$ is hyperbolic relative to $\pi(K)$.
\end{lemma}

\begin{proof}
This is achieved in the course of the proof of \cite[Theorem 3.1]{dimensions}: see \cite[Lemma 3.5]{dimensions}.
\end{proof}

\begin{lemma}
\label{CL vs T2}

Let $\Lambda < \Gamma$ be such that $(\Gamma, \Lambda, \Lambda)$ is a Cohen--Lyndon triple, and denote $\bar{\Gamma} \coloneqq \Gamma / \normal{\Lambda}$. Let $V$ be a $\bar{\Gamma}$-module. Then there is an exact sequence
\[\HHH^1(\Gamma; V) \to \HHH^1(\Lambda; V) \to \HHH^2(\bar{\Gamma}; V).\]
\end{lemma}

\begin{proof}
We apply the excision theorem for Cohen--Lyndon triples \cite[Theorem A]{CL2}: if $(\Gamma, \Lambda, N)$ is a Cohen--Lyndon triple, and we denote $\bar{\Lambda} \coloneqq \Lambda/N, \bar{\Gamma} \coloneqq \Gamma/\normal{N}$, then we have an isomorphism
\[\HHH^2(\Gamma, \Lambda; V) \cong \HHH^2(\bar{\Gamma}, \bar{\Lambda}; V),\]
for all $\bar{\Gamma}$-modules $V$. In our setting, $\Lambda = N$ so $\bar{\Lambda} = 1$ and $\HHH^2(\Gamma, \Lambda; V) \cong \HHH^2(\bar{\Gamma}; V)$. Now the exact sequence above is part of the long exact sequence for the pair $(\Gamma, \Lambda)$.
\end{proof}

\begin{proof}[Proof of \cref{one rel quotients}]
Let $\Gamma$ be an infinite relatively hyperbolic group with property $\T$, which automatically implies that $\Gamma$ is non-elementary. We first quotient out the maximal finite normal subgroup, to assume that $\Gamma$ has no non-trivial finite normal subgroups \cite[Lemma 5.10]{hull} (this introduces the possible finite kernel). \cref{find CL} gives an infinite cyclic subgroup $\Lambda$ such that $(\Gamma, \Lambda, \Lambda)$ is a Cohen--Lyndon triple; let $\bar{\Gamma} \coloneqq \Gamma/\normal{\Lambda}$, and let $V$ be a unitary $\bar{\Gamma}$-representation. Now $\HHH^1(\Gamma; V) = 0$ because $\Gamma$ has property $\T$, hence \cref{CL vs T2} gives an injection $\HHH^1(\Lambda; V) \hookrightarrow \HHH^2(\bar{\Gamma}; V)$. Because $V$ is a trivial $\Lambda$-module, and $\Lambda \cong \mathbb{Z}$, we see that $\HHH^1(\Lambda; V) \cong V \neq 0$.
We obtained $\bar{\Gamma}$ by adding finitely many relations, hence this construction preserves finite presentability.
\end{proof}

\begin{remark}
\label{L1}

The proof shows more generally that for every $\bar{\Gamma}$-module $V$:
\[\HHH^1(\Gamma; V) = 0 \quad \Rightarrow \quad \HHH^2(\bar{\Gamma}; V) \neq 0.\]
We used that $\Gamma$ has property $\T$ to have this for all unitary $\bar{\Gamma}$-representations, but using \cite{L1} we can similarly deduce it for all $L^1$-spaces with a $\bar{\Gamma}$-action. Higher vanishing with $L^1$-coefficients has recently gained attention because of its geometric implications: it forces actions on finite-dimensional contractible complexes of a given dimension to have a finite orbit. See \cite[Theorem 84]{ICM} for the precise statement, which is proved in upcoming work of Bader--Bader--Bader--Sauer \cite{L1:lattices}.
\end{remark}

\begin{remark}
Not all infinite quotients of a group with $\TT$ must fail to have $\TT$. For a normal subgroup $N < \Gamma$, denoting $\bar{\Gamma} \coloneqq \Gamma/N$ and letting $V$ be a $\bar{\Gamma}$-module, the inflation-restriction exact sequence gives:
\[\HHH^1(\Gamma; V) \to \HHH^1(N; V)^{\bar{\Gamma}} \to \HHH^2(\bar{\Gamma}; V) \to \HHH^2(\Gamma; V).\]
If we now assume that $V$ is a unitary $\bar{\Gamma}$-representation with no invariant vectors, and $\Gamma$ has property $\TT$, then we have an isomorphism $\HHH^2(\bar{\Gamma}; V) \cong \HHH^1(N; V)^{\bar{\Gamma}}$.
Because $V$ is a trivial $N$-module, this shows in particular that if $\HHH^1(N; \mathbb{Q}) = 0$ then $\bar{\Gamma}$ has property $\TT$. Property $\TTT$ in similarly inherited.
\end{remark}

\begin{remark}
Let $\Gamma$ be a group as in \cref{main}, and let $(\phi_n \colon \Gamma \to \Un)_n$ be an asymptotic homomorphism. \cite[Section 3]{DCGLT} associates a unitary $\Gamma$-representation $V$ and a class $\alpha \in \HHH^2(\Gamma; V)$ which vanishes if and only if $(\phi_n)_n$ has \emph{defect diminishing}; this holds for all asymptotic homomorphisms if and only if $\Gamma$ is Frobenius stable \emph{with a linear rate} \cite{defdim}. Since $\HHH^2(\Gamma; V) \neq 0$, it is tempting to say that $\Gamma$ is Frobenius stable with a superlinear rate. However, our non-vanishing result says nothing about $\alpha$.
\end{remark}

\pagebreak

\section{Post-scriptum: on AI benchmarks}

About six months after this paper first appeared on the arXiv, I received the following email from one of the authors of \cite{bench}.

\medskip

\small

[...] We are reaching out because we are developing a benchmark to evaluate the mathematical reasoning abilities of large language models, and your paper was selected as one of the source papers for this effort. [...]

The attached PDF contains one multiple-choice question based on your work. It includes the theorem statement, a proof sketch, the question, and the answer options. If you have the time and are willing to do so, we would be very grateful if you could read the PDF and help us assess whether the question is clear, faithfully reflects the theorem, and has a unique correct answer. [...]

We also want to emphasize that our intention is to better understand how AI systems can engage with mathematical reasoning, while fully respecting the originality and depth of the underlying mathematical work. If this use of AI in relation to mathematical research raises any concerns or feels inappropriate in any way, we sincerely apologize. [...]

\normalsize

\medskip

I replied to the email, twice, but have yet to receive an answer. This is perhaps due to the fact that the email itself was automated, as suggested by the announced proof sketch reading ``No expanded sketch found''.

Here I want to discuss the content of that PDF, because in a time where many news headlines give the perception that AI has ``solved'' much of mathematics, this experience highlights some of the fundamental shortcomings of this sort of endeavour. If this use of mathematics in relation to AI research feels inappropriate in any way, I sincerely apologise.

\begin{question}
\label{AI}
Let $G$ be a countable recursively presented group. Which existence statement holds?
\begin{enumerate}[(a)]
\item\label{A} There exists a finitely presented group $\Gamma$ into which $G$ embeds such that $\Gamma$ has property $\T$, is Frobenius stable, and does not have property $\TT$.
\item\label{B} There exists a finitely presented group $\Gamma$ into which $G$ embeds such that $\Gamma$ has property $\T$, is Frobenius stable, does not have property $\TT$, and satisfies $\HHH^2(\Gamma; V) \neq 0$ for every irreducible finite-dimensional unitary $\Gamma$-representation $V$, but $\HHH^2(\Gamma; W) = 0$ for some unitary $\Gamma$-representation $W$.
\item\label{C} There exists a finitely presented group $\Gamma$ into which $G$ embeds such that $\Gamma$ has property $\T$, is Frobenius stable, does not have property $\TT$, and in fact satisfies $\HHH^2(\Gamma; V) \neq 0$ for every unitary $\Gamma$-representation $V$.
\item\label{D} There exists a finitely presented group $\Gamma$ into which $G$ embeds such that $\Gamma$ has property $\T$, is Frobenius stable, does not have property $\TT$, and nevertheless satisfies $\HHH^2(\Gamma; V) = 0$ for every finite-dimensional unitary $\Gamma$-representation $V$.
\item\label{E} There exists a finitely presented group $\Gamma$ into which $G$ embeds such that $\Gamma$ has property $\T$, is Frobenius stable, does not have property $\TT$, and satisfies $\HHH^2(\Gamma; V) \neq 0$ for every non-trivial unitary $\Gamma$-representation $V$, but $\HHH^2(\Gamma; V) = 0$ for the trivial unitary representation.
\end{enumerate}
\end{question}

It seems clear that turning \cref{main} into \cref{AI} was itself the work of a large language model with no human supervision, as any human reading this would realise that option \eqref{C} is \cref{main} and option \eqref{A} is just a subset of the same statement, so they are obviously both correct.

Besides this, \cref{AI} fails to capture the very nature of the statement. Unambiguously identifying a unique correct answer would involve proving an existence result, but also four non-existence results. Perhaps a reminder is needed: \emph{a small perturbation of a true statement can still be true}, also in this case, see \cref{AI:answers}.

This is an illustrative example of the problematic nature of AI benchmarks that try to condense an arbitrary mathematical statement into one that can be easily algorithmically checked. I have contributed to benchmarks with questions of this form \cite{AI1, AI2}, but those were specifically designed (by humans) to fit the format. The vast majority of mathematics is not amenable to such simplifications. Scraping the arXiv for papers to autonomously distil into a multiple-choice question misses the point.

\begin{acks*}
I thank Alon Dogon, Tim Gehrunger, Simon Machado, Marco Moraschini, Anand Tadipatri, Henry Wilton and Julian Wykowski for insightful comments. I subscribe to the Leiden Declaration on Artificial Intelligence and Mathematics.
\end{acks*}

We end by proving that (at least) four out of the five proposed options in \cref{AI} are correct. This will use variations of the construction leading to \cref{main}.

\begin{proposition}
\label{AI:answers}
Options \eqref{A}, \eqref{B}, \eqref{C} and \eqref{D} in \cref{AI} are all correct.
\end{proposition}

\begin{proof}
As we already mentioned, options \eqref{A} and \eqref{C} follow directly from \cref{main}.

\medskip

Next, we prove that option \eqref{D} is correct. We will use that to every group $G$ one can associate a group $V(G)$ (the ``labelled Thompson group''), with the following properties: $V(G)$ is recursively presented if $G$ is, it has solvable word problem if $G$ does \cite{BHT, WWZZ}, and it is always acyclic \cite{palmer:wu}.

Let $G$ be a countable recursively presented group. We apply \cref{main} to obtain a finitely presented group $\Delta$ into which $V(G)$ embeds, such that $\Delta$ has property $\T$, is Frobenius stable, and $\HHH^2(\Delta; V) \neq 0$ for every unitary $\Delta$-representation $V$. Frobenius stability was established the following way: given a non-Frobenius approximable group $A$ with solvable word problem, we can construct $\Delta$ so that every non-identity element of $A$ normally generates $\Delta$, and so $\Delta$ has no non-trivial Frobenius approximable quotients by \cref{normalgen}. In the proof we chose $A$ to be a non-Frobenius approximable group $B$ from \cite{DCGLT}, here we instead take $A = V(B)$, which still has solvable word problem and is moreover acyclic.

Since $\Delta$ has $\T$, its abelianisation is finite, but since it has no non-trivial Frobenius approximable quotients, $\Delta$ must be perfect. Let $Z \coloneqq \HHH_2(\Delta; \Z)$, which is a finitely generated abelian group. By \cite[Theorem 6.9.5]{weibel}, we may consider the universal central extension
\[1 \to Z \to \Gamma \to \Delta \to 1,\]
which by \cite[Corollary 6.9.8]{weibel} satisfies $\HHH_i(\Gamma; \Z) = 0$ for $i \in \{ 1, 2 \}$. We will show that $\Gamma$ satisfies the conditions in option \eqref{D}.

First, $\Gamma$ is finitely presented as an extension of two finitely presented groups. Because $V(G)$ is acyclic and embeds into $\Delta$, there is a section $V(G) \to \Gamma$, and thus $G$ embeds into $\Gamma$. Similarly, since we chose $A$ to be acyclic, there is a section $\sigma \colon A \to \Gamma$. We claim that every non-identity element of $\sigma(A)$ normally generates $\Gamma$. Indeed, let $N$ be the normal closure of $\sigma(a)$ in $\Gamma$, for some $1 \neq a \in A$. Then the image of $N$ under the quotient $\Gamma \to \Delta$ is the normal closure of $1 \neq a \in A$, so it is equal to $\Delta$. This shows that $NZ = \Gamma$, so $Z$ surjects onto $\Gamma/N$, which is therefore abelian. Because $\Gamma$ is perfect, this is only possible if $N = \Gamma$. \cref{normalgen} now implies that $\Gamma$ has no non-trivial Frobenius approximable quotients, so it is Frobenius stable.

Property $\T$ for $\Gamma$ follows from \cite[Th{\'e}or{\`e}me 2.c.12]{dlHV}. For the failure of property $\TT$, let $V$ be a unitary $\Delta$-representation with $V^\Delta = 0$, and consider the pullback of $V$ to $\Gamma$. The inflation-restriction exact sequence includes
\[\HHH^1(Z; V)^\Delta \to \HHH^2(\Delta; V^Z) \to \HHH^2(\Gamma; V).\]
Since $Z$ acts trivially on $V$, we have an isomorphism of $\Delta$-modules $\HHH^1(Z; V) \cong \oplus_{\mathrm{rk}
(Z)} V$, so $\HHH^1(Z; V)^\Delta = 0$. Hence $\HHH^2(\Delta; V^Z) = \HHH^2(\Delta; V) \neq 0$ embeds into $\HHH^2(\Gamma; V)$.

Since finitely generated linear groups are residually finite \cite{malcev}, and $\Gamma$ has no finite quotients, the only finite-dimensional unitary $\Gamma$-representations are the trivial ones. Since $\HHH_2(\Gamma; \Z) = 0$, by the universal coefficient theorem $\HHH^2(\Gamma; V) = 0$ for every finite-dimensional unitary $\Gamma$-representation, and we conclude.

\medskip

Finally, we prove that option \eqref{B} is correct. Let $G$ be a countable recursively presented group. By \cref{main} there is a finitely presented group $\Delta$ into which $G$ embeds, such that $\Delta$ has property $\T$, is Frobenius stable, and $\HHH^2(\Delta; V) \neq 0$ for every unitary $\Delta$-representation $V$. Again, we use that the Frobenius stability of $\Delta$ follows from the stronger fact that it has no non-trivial Frobenius approximable quotients.

Now consider the group $\Gamma = \Delta_1 \times \Delta_2$, where each $\Delta_i$ is a copy of $\Delta$. We will show that $\Gamma$ satisfies the conditions in option \eqref{B}. First, it is a finitely presented group into which $G$ embeds, and it has property $\T$ by \cite[Proposition 1.9]{dlHV}. If $\Gamma/N$ is a Frobenius approximable quotient, then the image of $\Delta_i$ in $\Gamma/N$ is a Frobenius approximable quotient of $\Delta_i$; this shows that $\Gamma$ has no non-trivial Frobenius approximable quotients, and so it is Frobenius stable, and has no non-trivial finite quotients. Again, this implies that the only finite-dimensional unitary $\Gamma$-representations are the trivial ones.

We claim that if $V$ is a unitary $\Gamma$-representation such that $V^{\Delta_1} \neq 0$, then $\HHH^2(\Gamma; V) \neq 0$. Indeed the inflation-restriction exact sequence includes
\[\HHH^1(\Delta_1; V)^{\Delta_2} \to \HHH^2(\Delta_2; V^{\Delta_1}) \to \HHH^2(\Gamma; V).\]
Since $\Delta_1$ has $\T$, this shows that $\HHH^2(\Delta_2; V^{\Delta_1}) \neq 0$ embeds into $\HHH^2(\Gamma; V)$. Taking $V$ to be the pullback of a unitary $\Delta_2$-representation with $V^{\Delta_2} = 0$, we see that $\Gamma$ does not have $\TT$. Taking $V$ to be $\C$, that is the unique irreducible finite-dimensional unitary $\Gamma$-representation, we see that $\HHH^2(\Gamma; \C) \neq 0$.

It remains to show that there exists a unitary $\Gamma$-representation $W$ such that $\HHH^2(\Gamma; W) = 0$. We claim that this holds for $W = \ell^2(\Gamma) = \ell^2(\Delta_1) \otimes \ell^2(\Delta_2)$. Indeed, by K{\"u}nneth's formula
\[\HHH^2(\Gamma; W) = \sum\limits_{i+j = 2} \HHH^i(\Delta_1; \ell^2(\Delta_1)) \otimes \HHH^j(\Delta_2; \ell^2(\Delta_2)),\]
and this vanishes since $\HHH^0(\Delta; \ell^2(\Delta)) = 0$ (because $\Delta$ is infinite) and $\HHH^1(\Delta; \ell^2(\Delta)) = 0$ (because $\Delta$ has $\T$).
\end{proof}

\pagebreak

\footnotesize

\bibliographystyle{amsalpha}
\bibliography{ref}

\vspace{0.5cm}

\normalsize

\noindent{\textsc{Department of Pure Mathematics and Mathematical Statistics, University of Cambridge, UK}}

\noindent{\textit{E-mail address:} \texttt{ff373@cam.ac.uk}}

\end{document}